\documentclass[12pt,reqno]{amsart}
\usepackage{amssymb, amsfonts, amsbsy, latexsym, epsfig, color}

\newtheorem{Thm}{Theorem}[section]

\newtheorem{corollary}[Thm]{Corollary}
\newtheorem{lemma}[Thm]{Lemma}
\newtheorem{proposition}[Thm]{Proposition}
\newtheorem{definition}[Thm]{Definition}
\newtheorem{remark}[Thm]{Remark}

\newtheorem{example}[Thm]{Example}

\newtheorem{theorem}[Thm]{Theorem}

\newcommand{\inner}[2]{\langle #1,#2 \rangle}

\newcommand{\abs}[1]{\left| #1 \right|}

\newcommand{\absipt}[1]{\left| \left\langle #1 \right\rangle \right|}
\newcommand{\ipt}[1]{\left\langle #1 \right\rangle}

\def\ldots{\mathinner{\ldotp\ldotp\ldotp}}
\def\ldots{\mathinner{\cdotp\cdotp\cdotp}}
\newcommand{\norm}[1]{\|#1\|}

\def\RR{\mathbb{R}}

\def\CC{\mathbb{C}}

\def\HH{\mathbb H}

\newcommand{\set}[1]{\left\{#1\right\}}%

\begin{document}

\title{Norm Retrieval and Phase Retrieval by Projections}
\author[Casazza, Ghoreishi, Jose, Tremain
 ]{Peter G. Casazza,
 Dorsa Ghoreishi, Shani Jose, Janet C. Tremain}
\address{Department of Mathematics, University
of Missouri, Columbia, MO 65211-4100}

\email{‎Casazzap@missouri.edu; dorsa.ghoreishi@gmail.com;}
\email{ shanijose@gmail.com; Tremainjc@missouri.edu}

\thanks{The first, second and fourth authors were supported by
 NSF DMS 1609760; NSF ATD 1321779; and ARO  W911NF-16-1-0008.
 Part of this research was carried out while the first and
 fourth authors were visiting the Hong Kong University of
 Science and Technology on a grant from (ICERM) Institute for computational and experimental research in
Mathematics.
}

\begin{abstract}
We make a detailed study of norm retrieval.  We give several
classification theorems for norm retrieval and give a large
number of examples to go with the theory.  One consequence is
a new result about Parseval frames:  If a Parseval frame is
divided into two subsets with spans $W_1,W_2$ and
$W_1 \cap W_2=\{0\}$, then $W_1
\perp W_2$.    
\end{abstract}
\maketitle

\section{Introduction}

Signal reconstruction is an important problem in engineering and has a wide variety of applications. Recovering signals when there is partial loss of information is a significant challenge. Partial loss of phase information occurs in application areas such as speech recognition~\cite{BeRi99,RaJu93,ReBlScCa004}, and optics applications such as X-ray crystallography~\cite{BaMn86,Fi78,Fi82}, and there is a need to do phase retrieval efficiently. The concept of \textit{phase retrieval} for Hilbert space frames was introduced in 2006 by Balan, Casazza, and Edidin~\cite{BaCaEd06}, and since then it has become an active area of research in signal processing and harmonic analysis. 

Phase retrieval has been defined for vectors as well as for projections and in general deals with recovering the phase of a signal given its intensity measurements from a redundant linear system. Phase retrieval by projections, where the signal is projected onto some higher dimensional subspaces and has to be recovered from the norms of the projections of the vectors onto the subspaces, appears in real life problems such as crystal twinning~\cite{Dr010}. We refer the reader to~\cite{CCPW} for a detailed study of phase retrieval by projections.

Another related problem is that of phaseless reconstruction, where the unknown signal is reconstructed from the intensity measurements. Recently, the two terms phase retrieval and phaseless reconstruction were used interchangeably. However, it is not clear from their respective definitions how these two are equivalent. Recently, in~\cite{SaCa016} the authors proved the equivalence of phase retrieval and phaseless reconstruction in real as well as in complex case. Due to this equivalence, in this paper, we restrict ourselves to proving results regarding phase retrieval. Further, a weaker notion of phase retrieval and phaseless reconstruction was introduced in~\cite{SaPeDoShJa16}.

In this work, we consider the notion of \textit{norm retrieval} which was recently introduced by Bahmanpour et.al. in~\cite{BCCJW}, and is the problem of retrieving the norm of a vector given the absolute value of its intensity measurements. Norm retrieval arises naturally from phase retrieval when one utilizes both a collection of subspaces and their orthogonal complements. Here we study norm retrieval and certain classifications of it. We use projections to do norm retrieval and to extend certain results from~\cite{HF} for frames. We provide a complete classification of subspaces of $\RR^N$ which do norm retrieval. Various examples for phase and norm retrieval by projections are given. Further, a classification of norm retrieval using Naimark's theorem is also obtained.

We organize the rest of the paper as follows. In Section~\ref{s:prilim}, we include basic definitions and results of phase retrieval. Section~\ref{s:beginning_normret} introduces the norm retrieval and properties. Section~\ref{s:phaseret_normret} provides the relationship between phase and norm retrieval and related results. Detailed classifications of vectors and subspaces which do norm retrieval are provided in Section~\ref{s:classifications_normret}.

\section{Preliminaries}\label{s:prilim}

We denote by $\HH^N$ a $N$ dimensional real or complex Hilbert space, and we write $\RR^N$ or $\CC^N$ when it is necessary to differentiate between the two explicitly. Below, we give the definition of a frame in $\HH^N$.

\begin{definition}\label{D:frame}
A family of vectors $\Phi=\{\phi_i\}_{i=1}^M$ in $\HH^N$ is a {\bf frame} if there are constants $0<A\leq B<\infty$ so that for all $x\in \HH^N$,
\begin{equation}\label{E:defn_frame}
A \|x\|^2 \leq \sum_{i=1}^n |\langle x, \phi_i\rangle|^2\leq B\norm{x}^2.
\end{equation}
\end{definition}

The following definitions and terms are useful in the sequel.
\begin{itemize}
\item The constants $A$ and $B$ are called the \textbf{lower and upper frame bounds} of the frame, respectively.
\item If $A=B$, the frame is called an \textbf{$A$-tight frame} (or a tight frame). In particular, if $A=B=1$, the frame is called a \textbf{Parseval frame}.
\item $\Phi$ is an \textbf{equal norm frame} if $\norm{\phi_i}=\norm{\phi_j}$ for all $i, j$ and is called a \textbf{unit norm frame} if $\norm{\phi_i}=1$ for all $i=1, 2, \ldots n$.
\item If, only the right hand side inequality holds in~\eqref{E:defn_frame}, the frame is called a \textbf{B-Bessel 
family with Bessel bound $B$}.
\end{itemize}

Note that in a finite dimensional setting, a frame is a spanning set of vectors in the Hilbert space. We refer to~\cite{CaLy016} for an introduction to Hilbert space frame theory and applications.


Let $\Phi=\{\phi_i\}_{i=1}^N$ be a frame in $\HH^N$. The \textbf{analysis operator} associated with $\Phi$ is defined as the operator $T:\HH^N \rightarrow \ell_2^M$ to be
\[Tx=\sum_{i=1}^M\ipt{x, \phi_i}e_i = \{\ipt{x, \phi_i}\}_{i=1}^M, \mbox{ for all } x\in \HH^N.\]
Here, $\{e_i\}_{i=1}^M$ is understood to be the natural orthonormal basis for $\ell_2^M$. The adjoint $T^*$ of the analysis operator $T$ is called the~\textbf{synthesis operator} of the frame $\Phi$. It can
be shown that $T^*(e_i)=\phi_i.$

The \textbf{frame operator} for the frame $\Phi$ is defined as $S:T^*T:\HH^N\rightarrow \HH^N.$ That is,
\[Sx=T^*T(x)=\sum_{i=1}^M\ipt{x, \phi_i}\phi_i
\mbox{ for all }x \in \HH^N.\]

Note that the frame operator $S$ is a positive, self-adjoint and invertible operator satisfying the operator inequality $AI \leq S\leq BI$, where $A$ and $B$ are the frame bounds and $I$ denotes the identity on $\HH^N$. Frame operators play an important role since they are used to reconstruct the vectors in the space. To be precise, any $x\in \HH^N$ can be written as
\begin{equation}\label{E:reconstruct_formula}
x=SS^{-1}x=S^{-1}Sx=\sum_{i=1}^M \ipt{S^{-1}x, \phi_i}\phi_i = \sum_{i=1}^M \ipt{x, S^{-1}\phi_i}\phi_i.
\end{equation}

The frame operator of a Parseval frame is the identity operator. Thus, if $\{\phi_i\}_{i=1}^M$ is a Parseval frame, it follows from equation~\eqref{E:reconstruct_formula} that
\[x= \sum_{i=1}^M\ipt{x, \phi_i}\phi_i, ~~x\in \HH^N.\]

We concentrate on norm retrieval and its classifications in this paper. 
We now see the basic definitions of phase retrieval formally, starting with phase retrieval by projections. Throughout the paper, the term projection is used to describe orthogonal projection (orthogonal idempotent operator) onto subspaces.

\begin{definition}\label{D:Phase retrieval by projections}
Let $\{W_i\}_{i=1}^M$ be a collection of subspaces in $\HH^N$ and let $\{P_i\}_{i=1}^M$ be the projections onto each of these subspaces. We say that $\{W_i\}_{i=1}^M$ (or $\{P_i\}_{i=1}^M$) yields {\bf phase retrieval}  if for all $x, y \in \HH^N$ satisfying $\|P_ix\|=\|P_iy\|$ for all $i = 1,2, \ldots,M$ then $x=cy$ for some scalar $c$ such that $|c|=1$
\end{definition}

Phase retrieval by vectors is a particular case of the above.

\begin{definition}
Let $\Phi=\{\phi_i\}_{i=1}^M \in \it{H}^N$ be such that for $x, y\in \it{H}^N$
\[ |\langle x,\phi_i\rangle|=|\langle y,\phi_i\rangle|, \mbox{ for all }i=1,2,\ldots,M. \]
$\Phi$ yields
\textbf{phase retrieval} with respect to an orthonormal basis $\{e_i\}_{i=1}^N$ if there is a $|\theta|=1$ such that $x_i=c y_i$, for all $i=1,2,\ldots,N$, where $x_i = \langle x,e_i\rangle$.
%
\end{definition}

Orthonormal bases fail to do phase retrieval, since in any given orthonormal basis, the corresponding coefficients of a vector are unique.
One of the fundamental properties to identify the minimum number of vectors required to do phase retrieval is the complement property.

\begin{definition}[\cite{BaCaEd06}]\label{D:complement_prop}
A frame $\Phi=\{\phi_i\}_{i=1}^M $in $\HH^N$ satisfies the {\bf complement property} if for all subsets ${I}\subset\{1, 2, \ldots, M\}$, either $\{\phi_i\}_{i\in I}$ or $\{\phi_i\}_{i\in I^c}$ spans the whole space $\HH^N$.
\end{definition}

It is proved in~\cite{BaCaEd06} that phase retrieval is equivalent to the complement property in $\RR^N$. Further, it is proven that a generic family of $(2N-1)$-vectors in $\mathbb{R}^N$ does phase retrieval, however no set of $(2N -2)$-vectors can. Here, generic 
 refers to an open dense set in the set of $(2N -1)$-element frames in $\HH^N$. Full spark is another important notion of vectors in frame theory. A formal definition is given below:

\begin{definition}\label{D:full_Spark}
Given a family of vectors $\Phi=\{\phi_i\}_{i=1}^M$ in $\HH^N$, the {\bf spark} of $\Phi$ is defined as the cardinality of the smallest linearly dependent subset of $\Phi$. When spark$(\Phi) = N + 1$, every subset of size $N$ is linearly independent, and in that case, $\Phi$ is said to be {\bf full spark}.
\end{definition}

Note from the definitions that full spark frames with $M\ge 2N-1$ have the complement property and hence do phase retrieval. Moreover, if $M=2N-1$ then the complement property clearly implies full spark.

Next result, known as Naimark's theorem, characterizes Parseval frames in a finite dimensional Hilbert space. This theorem facilitates a way to construct Parseval frames, and crucially it is the only way to obtain Parseval frames. Later, we use this to obtain a classification of frames which do norm retrieval. The notation $[M]=\{1, 2, \ldots, M\}$ is used throughout the paper.

\begin{theorem}[Naimark's Theorem]~\cite{CaGi013}
A frame $\{\phi_i\}_{i=1}^M$ is a Parseval frame for $\RR^N$ if and only if $\RR^N \subset \ell_2^M$ with orthonormal basis $\{e_i\}_{i=1}^M$ so that the orthogonal projection $P$ onto $\RR^N$ satisfies:  $Pe_i=\phi_i$ for every $i\in [M]$.
\end{theorem}

\section{Beginnings of Norm Retrieval}\label{s:beginning_normret}

In this section, we provide the definition of norm retrieval along with certain related results, and pertinent examples.

\begin{definition}\label{D:norm retrieval by projections}
Let $\{W_i\}_{i=1}^M$ be a collection of subspaces in $\HH^N$ and let $\{P_i\}_{i=1}^M$ be the orthogonal projections onto each of these subspaces. We say that $\{W_i\}_{i=1}^M$ (or $\{P_i\}_{i=1}^M$) yields \textbf{norm retrieval}  if for all $x, y \in \HH^N$ satisfying $\|P_ix\|=\|P_iy\|$ for all $i = 1,2, \ldots,M$ then $\norm{x}=\norm{y}$.

In particular, a set of vectors $\{\phi_i\}_{i=1}^M$ in $\HH^N$ does norm retrieval, if for $x, y \in \HH^N$ satisfying $\absipt{x, \phi_i}=\absipt{y, \phi_i}$ for all $i = 1,2, \ldots,M$ then $\norm{x}=\norm{y}$.
\end{definition}

\begin{remark}
It is immediate that a family of vectors doing phase retrieval
does norm retrieval.
\end{remark}

An obvious choice of vectors which do norm retrieval are orthonormal bases. For, let $\{e_i\}_{i=1}^N$ be an orthonormal basis in $\HH^N$. Now, for $x\in \HH^N$, $\absipt{x, \phi_i}=\absipt{x, e_i}=\abs{x_i}.$ Thus
\[\sum_{i=1}^N\absipt{x, \phi_i}^2=\sum_{i=1}^N\abs{x_i}^2=\norm{x}^2.\]

The following theorem provides a sufficient condition under which the subspaces spanned by the canonical basis vectors do norm retrieval.
\begin{theorem}
Let $\{e_i\}_{i=1}^N$ be an orthonormal basis in $\HH^N$. Let $\{W_j\}_{j=1}^k$ be subspaces of $\HH^N$ where each $W_j=span\{e_i\}_{i\in I_j}$, $I_j\subseteq [N]$. If there exists $m$ such that  for all $j$, $|\{j: e_i \in W_j\}|=m$, then $\{W_j\}_{j=1}^k$ does norm retrieval.
\end{theorem}

\begin{proof}
Let $P_j$ be orthogonal projections onto $W_j$, for all $j$. Now, by assumption, we have

\[\sum_{j=1}^k \norm{P_jx}^2 = \sum_{j=1}^k\sum_{i\in I_j}\absipt{x, e_i}^2  = m \sum_{j=1}^N \absipt{x, e_j}^2 = m\norm{x}^2. \]
\end{proof}

It is easy to see that tight frames do norm retrieval. 

\begin{theorem}Tight frames do norm retrieval.
\end{theorem}

\begin{proof}et $\{\phi_i\}_{i=1}^M$ in $\HH^N$ be an A-tight frame. Now, if
\[ |\langle x,\phi_i\rangle|=\langle y,\phi_i\rangle|, \mbox{ for all }i=1,2,\ldots,M,\]
then
\[ A\|x\|^2 = \sum_{i=1}^M|\langle x,\phi_i\rangle|^2 = \sum_{i=1}^M|\langle y,\phi_i\rangle|^2 = A \|y\|^2.\]
\end{proof}

Observe that if $\{\phi_i\}_{i=1}^M \in \HH^N$ does norm retrieval so does $\{\phi_i\}_{i=1}^M\cup\{\psi_j\}_{j=1}^K$ for any $\psi_j \in \HH^N$. This is generalized in the following proposition.

\begin{proposition}
If $\lbrace P_i\rbrace_{i=1}^M$ does norm retrieval, then so does $\lbrace P_i\rbrace_{i=1}^M\cup\lbrace Q_i\rbrace_{i=1}^K$ for any projections $Q_i$. In particular, if a frame $\Phi=\{\phi_i\}_{i=1}^M$ contains an orthonormal basis, then it does norm retrieval.  Moreover, in this case, $\{\phi_i^{\perp}\}_{i=1}^M$ does norm retrieval.
\end{proposition}

\begin{proof}
Let $\{e_i\}_{i=1}^N$ be an orthonormal basis for $\HH^N$ and let $P_i$ be the projections onto ${\phi_i}^\perp$, for each $i$. Given $x\in \HH^N$, we have
\[\sum_{i=1}^N\|P_ix\|^2=\sum_{i=1}^N\sum_{j\neq i}\absipt{x, e_j}^2=(N-1)\sum_{i=1}^N \absipt{x, e_j}^2=(N-1)\norm{x}^2. \]
\end{proof}

The above proposition does not hold if the number of hyperplanes is strictly less than $N$. This is proved in the next theorem.

\begin{theorem}
If $\lbrace \varphi_i\rbrace_{i=1}^N$ is an orthonormal basis for $\RR^N$ then $\lbrace W_i\rbrace _{i\in I}$ where $W_i=\varphi_i^\perp$ cannot do norm retrieval for $I\subseteq [N-1]$. 
\end{theorem}

\begin{proof} Without loss of generality consider the collection $\lbrace W_i\rbrace _{i=1}^{N-1}$ (for $N>2$). Now, let $x=\sum_{i=1}^N\varphi_i$ and $y=\sqrt{\frac{N-1}{N-2}}\sum_{i=1}^{N-1}\varphi_i$ so that $\|P_jx\|^2=\sum_{i\neq j}|\varphi_i|^2=N-1$ and $\|P_j y\|^2=\frac{N-1}{N-2}\sum_{\substack{i=1 \\i\neq j}}^{N-1}|\varphi_i|^2$. Thus, $\norm{P_j x}^2=\norm{P_jy}^2$. However $\|x\|^2=N$ and $\|y\|^2=\frac{(N-1)^2}{N-2}$ which proves the theorem.\\
\vspace{.1 in}


\end{proof}

Now, we strengthen the above result by not requiring the vectors to be orthogonal. To prove this, we need the following lemma.

\begin{lemma} \label{IndLem}
If $\{\phi_i\}_{i=1}^N$ are independent vectors in $\RR^N$, then there is a vector $\phi \in \RR^N$ satisfying:
\[ |\langle \phi,\phi_i\rangle|=c \not= 0,\mbox{ for all }
i=1,2,\ldots,N.\]
\end{lemma}

\begin{proof}
We do this by induction on $N$ with the case $N=2$ obvious. So assume this holds for $N-1$.  Given $\{\phi_i\}_{i=1}^N$, we can find a 
$\phi \in span\ \{\phi_i\}_{i=1}^{N-1}$ with $\|\phi\|=1$ and satisfying
\[ |\langle \phi,\phi_i\rangle|=c\not=0,\mbox{ for all } i=1,2,\ldots,N-1.\]

Choose $\psi \perp span\ \{\phi_i\}_{i=1}^{N-1}$ and note that linear independence of the $\phi_i$ implies
\[|\langle \psi ,\phi_N\rangle|\not= 0.\]

Consider $\phi+\lambda \psi$.  For $i=1,2,\ldots,N-1$,
\begin{align*}
 |\langle \phi+\lambda \psi,\phi_i\rangle|&=|\langle \phi,\phi_i\rangle
 + \lambda \langle \psi,\phi_i\rangle|\\
 &= |\langle \psi,\phi_i\rangle|\\
 &= c
 \end{align*}

Also,
 \[\langle \phi+\lambda \psi,\phi_N\rangle
 = \langle \phi,\phi_N\rangle + \lambda \langle \psi,\phi_N\rangle.\]

As $\lambda$ varies from $-\infty$ to $+\infty$, the right hand side varies from $-\infty$ to $+\infty$ and for some $\lambda$, we have
\[ |\langle \phi,\phi_N\rangle + \lambda \langle \psi,\phi_N\rangle|=c.\]

\end{proof}

\begin{proposition}
If $\{\phi_1,\ldots,\phi_{N-1}\} \in R^N$ are independent and unit norm and $\{W_i\}=\{\phi_i^{\perp}\}$, for all $i=1,2,\ldots,N-1$,
then $\{W_i\}_{i=1}^{N-1}$ cannot do norm retrieval.
\end{proposition}

\begin{proof}
Let $P_i$ be the projection onto $W_i$ and choose \[x \in \cap_{i=1}^{N-1}W_i,\mbox{ with }\|x\|=1.\]

By the assumption, there is a vector $\phi\in span\ \{\phi_i\}_{i=1}^{N-1},\mbox{ with } \|\phi\|=1,$ and $|\langle \phi,\phi_i\rangle|=c\not= 0, \mbox{ for all } i=1,\ldots,N-1.$ Let $y=\lambda x+\mu \phi$, where $\lambda^2+(1-c^2)\mu^2=1.$ Note that $x\perp \phi_i$ for all $i$ implies that $\phi \perp x$, and so $\|y\|^2 = \lambda^2+\mu^2 \not= 1.$

Now, for all $i=1,2,\ldots,N-1$,
\begin{align*}
\|P_iy\|^2&= \|y\|^2-|\langle y,a_i\rangle|^2\\
&= \lambda^2+\mu^2 - \mu^2c^2\\
&= \lambda^2+(1-c^2)\mu^2\\
&= 1\\
&= \|x\|^2\\
&= \|P_ix\|^2.
\end{align*}

But $\|x\|^2=1$ while $\|y\|^2 \not= 1$, and so norm retrieval fails.
\end{proof}

However, in the following theorem, we show that three proper subspaces of codimension one can do norm retrieval in $\RR^N$.

\begin{theorem}
In $\RR^N$ three proper subspaces of codimension one can do norm retrieval.
\end{theorem}

\begin{proof}
	Let $\{e_i\}_{i=1}^N$ be an orthonormal basis for $\RR^N$. Let
		$$\phi_1=e_1\qquad\phi_2=e_2\qquad \phi_3=(e_1-e_2)/\sqrt{2}$$
	We claim $\{\phi_i^\perp\}_{i=1}^3$ does norm retrieval. Let $P_i$ be the orthogonal projection onto $\phi_i^\perp$.
	Let $x=(a_1, \ldots, a_N)$. We then have that
	$$||P_1x||^2=a_2^2+\sum_{k=3}^Na_k^2, \qquad||P_2x||^2=a_1^2+\sum_{k=3}^Na_k^2$$
	$$||P_3x||^2=\left(\frac{a_1+a_2}{\sqrt{2}}\right)^2+\sum_{k=3}^Na_k^2=\frac{a_1^2+2a_1a_2+a_2^2}{2}+
	\sum_{k=3}^N a_k^2$$

	Case 1: If $a_1=0$ or $a_2=0$, we know that $||x||^2=||P_1x||^2$ or $||x||^2=||P_2x||^2$ respectively.

	Case 2: Assume both $a_1\neq 0$ and $a_2\neq 0$. We then know both of the equalities below:
	$$-\frac{(a_1+a_2)^2}{2}\cdot\frac{1}{a_2^2}||P_1x||^2+||P_3x||^2=c\sum_{k=3}^N a_k^2$$
	$$-\frac{(a_1+a_2)^2}{2}\cdot\frac{1}{a_1^2}||P_2x||^2+||P_3x||^2=d\sum_{k=3}^Na_k^2$$
	where
	$$c=-\frac{(a_1+a_2)^2}{2a_2^2}+1\qquad\text{and}\qquad d=-\frac{(a_1+a_2)^2}{2a_1^2}+1$$
	If either $c$ or $d$ is nonzero, then the proof is complete as in that case, we can express $||x||^2$ as a linear combination
	of $||P_1x||^2$, $||P_2x||^2$, and $||P_3x||^2$.

    Now, suppose that $c=d=0$. If $c=0$, then
	$(a_1+a_2)^2=2a_2^2$ and if $d=0$, then $(a_1+a_2)^2=2a_1^2$. This implies that
	$$2(a_1+a_2)^2=2a_1^2+2a_2^2$$
	which holds only if either $a_1$ or $a_2$ or both is zero which contradicts our assumption.
\end{proof}

It follows that in $\RR^3$, two 2-dimensional subspaces cannot
do norm retrieval but three 2-dimensional subspaces can do 
norm retrieval.

\begin{proposition}
For every $K\leq N$, there exist subspaces $\lbrace W_i\rbrace_{i=1}^{K+1}$ of $\HH^N$ which do norm retrieval and $\lbrace W_i^{\perp}\rbrace_{i=1}^{K+1}$ span a $K$ dimensional space.
\end{proposition}
\begin{proof} Choose an orthonormal basis of $\HH^N$, say $\lbrace e_i\rbrace_{i=1}^N$. Let $W_1=\text{span} \lbrace e_i\rbrace_{i=1}^{N-K}$ and $W_i=\text{span}\lbrace W_1, e_{N-K+i-1}\rbrace$ for all $2\leq i\leq K+1$. If $x=\sum_{j=1}^N a_j e_j$, then $\|P_1x\|^2=\sum_{j=1}^{N-K} |a_j|^2$ and $\|P_i x\|^2=\sum_{j=1}^{N-K}|a_j|^2+|a_{N-K+i-1}|^2$ for $2\leq i \leq K+1$. Therefore $\|x\|^2=\sum_{i=2}^{K+1} \|P_i x\|^2- (K-1)\|P_1 x\|^2$. Since $W_i^{\perp}\subseteq W_1^{\perp}$ for all $i$, it's clear that $\lbrace W_i^{\perp}\rbrace_{i=1}^{K+1}$ is spanned by $\{e_i\}_{i=M-K+1}^N$, which has dimension $K$ .
\end{proof}

The following proposition shows a relationship between subspaces doing norm retrieval and the sum of the dimensions of the subspaces. The importance of this proposition is that we are looking for conditions on subspaces to do norm retrieval. To do so, the dimension of the subspaces is one of the tools we have.

\begin{proposition}
If $\{W_i\}_{i=1}^M$ in $\mathbb{R}^N$ does norm retrieval then $\sum_{i=1}^M \dim W_i\geq N$. Moreover, if $\sum_{i=1}^M k_i=LN$ then there exist $\lbrace W_i\rbrace_{i=1}^M$ doing norm retrieval where $\dim W_i=k_i$ for each $1\leq i \leq M$.
\end{proposition}
\begin{proof}
If $\sum_{i=1}^M \dim W_i < N$ then we may pick non-zero $x\perp W_i$ for each $i$ so that $\|P_i x\|=0$ for all $i$ and therefore $\lbrace W_i\rbrace_{i=1}^M$ fails norm retrieval.

For the moreover part, let $\lbrace g_i\rbrace_{i=1}^N$ be an orthonormal basis. We represent this basis $L$-times as
a multiset:
\[ \{\phi_i\}_{i=1}^{LN}=:\{g_1,\ldots,g_N,g_1,\ldots,g_N,\ldots,g_1,\ldots,g_N\},\]
and index it as:  $\{e_i\}_{i=1}^{LN}$.
We may pick a partition of $[LN]$ in the following manner:
\[ I_1=\{1,2,\ldots,k_1\},\ I_2=\{k_1+1,\ldots,k_1+k_2\},\
I_3= \{k_1+k_2+1,\ldots,k_1+k_2+k_3\}, \ldots.\] Now define $W_i=$ span $\lbrace e_j\rbrace_{j\in I_i}$ with projection $P_i$. Then if $x=\sum_{j=1}^N a_je_j$ then
\[ \sum_{i=1}^M\|P_ix\|^2=\sum_{i=1}^M\sum_{j\in I_i}|a_j|^2=L\sum_{j=1}^N|a_j|^2=L\|x\|^2.\] Hence the result.
\end{proof}

As we have seen, the above proposition may fail if $\sum_{i=1}^M k_i\neq LN$.

\section{Phase retrieval and Norm Retrieval}\label{s:phaseret_normret}

In this section, we provide results relating phase retrieval and norm retrieval. The following theorem of Edidin~\cite{Ed015} is significant in phase retrieval as it gives a necessary and sufficient condition for subspaces to do phase retrieval.

\begin{theorem}[~\cite{Ed015}]\label{T:Edidns_theorem}
A family of projections $\{P_i\}_{i=1}^M$ in $\RR^N$ does phase retrieval if and only if for every $0\not= x\in \RR^N$, the vectors $\{P_ix\}_{i=1}^M$ span the space.
\end{theorem}

\begin{corollary}Let $\lbrace W_i\rbrace_{i=1}^M$ be a collection of subspaces of  $\RR^N$ with $P_i$ denoting the projection onto $W_i$ for each $1\leq i\leq M$. If $\lbrace W_i\rbrace_{i=1}^M$ does phase retrieval in $\RR^N$ then for every $I\subset [M]$ with $|I|\leq N-2$, the collection $\lbrace W_i^{\perp}\rbrace_{i\in I^c}$ spans $\RR^N$.
\end{corollary}

\begin{proof}
If not, pick non-zero $x\perp W_i^{\perp}$ for all $i\in I^c$. This implies  $x\in\cap_{i\in I^c} W_i$ and therefore $\lbrace P_i (x)\rbrace_{i=1}^N$ contains at most $N-1$ distinct vectors and can not span $\RR^N$. This contradicts the theorem~\ref{T:Edidns_theorem}.
\end{proof}

\begin{corollary}
If $\{W_i\}_{i=1}^M$ in $\HH^N$ does phase retrieval, then $\{W_i^{\perp}\}_{i=1}^M$ spans the space.
\end{corollary}

\begin{proof}
If $(W_i^{\perp})$ does not span, then there exists $0\not= x \in \cap W_i$. So $P_ix=x$ for all $i=1, 2, \ldots, M$, and so $\{P_i(x)\}$ does not span. Thus, by Theorem~\ref{T:Edidns_theorem}, $(W_i)$ does not do phase retrieval.
\end{proof}

The following example shows that it is possible for subspaces to do norm retrieval even if $\{W_i^{\perp}\}$ do not span the space {which we see as one of main differences between phase retrieval and norm retrieval.}

\begin{example}
Let $\{e_i\}_{i=1}^3$ be a orthonormal basis for $\RR^3$, then let
\begin{alignat*}{2}
\\ &W_1=span\{e_1,e_2\} &\qquad & {W_1}^\perp= span \{ e_3\}
\\ &W_2=span\{e_2,e_3\} &\qquad & {W_2}^\perp= span \{ e_1\}
\\ &W_3=span\{e_2\} &\qquad & {W_3}^\perp= span \{ e_1,e_3\}
\end{alignat*}

Then, $\{W_i\}_{i=1}^3$ does norm retrieval since $\|x\|^2=\|P_1 x\|^2+\|P_2 x\|^2-\|P_3 x\|^2$. But $\{W_i\}^\perp, ~i=1, 2, 3$ do not span $\RR^3$.

Note that if $W_1=\HH^N$, then $\{W_1\}$ itself does norm retrieval while $W_1^{\perp}= \{0\}$.
\end{example}

Any collection of subspaces which does phase retrieval yields norm retrieval, which follows from the definitions. However, the converse need not hold true always. For instance, any orthonormal basis does norm retrieval in $\RR^N$. But it has too few vectors to do phase retrieval as it requires at least $2N-1$ vectors to do phase retrieval in $\RR^N$.

Given subspaces $\{W_i\}_{i=1}^M$ of $\HH^N$ which yield phase retrieval, it is not necessarily true that $\{W_i^\perp\}_{i=1}^M$ do phase retrieval. The following result proves that norm retrieval is the condition needed to pass phase retrieval to orthogonal complements. Though the result is already proved in~\cite{BCCJW}, we include it here for completeness.

\begin{lemma}~\label{L:perpnormretrieval}
Suppose subspaces $\{W_i\}_{i=1}^M$, with respective projections $\{P_i\}_{i=1}^M$, does phase retrieval. Then $\{W_i^{\perp}\}_{i=1}^M$ does phase retrieval if and only if $\{W_i^{\perp}\}_{i=1}^M$ does norm retrieval.
\end{lemma}

\begin{proof}
Assume that $\|(I-P_i)x=\|(I-P_i)y$ for all $i=1,2,\ldots,M$
and $\{P_i\}_{i=1}^M$ does norm retrieval.  I.e. $\|x\|=\|y\|$.
Then
\[\|(I-P_i)(x)\|^2=\|x\|^2-\|P_i x\|^2=\| y\|^2-\|P_i y\|^2=\| (I-P_i)(y)\|^2.\]
Since $\|x\|=\|y\|$, we have
\[ \|P_ix\|=\|P_iy\|\mbox{ for all }i=1,2,\ldots,M.\]
Since $\{P_i\}_{i=1}^M$ does phase retrieval, it follows that
$x=cy$ for some $|c|=1$.

The other direction of the theorem is clear.
\end{proof}

Next is an example of a family of subspaces $\{W_i\}_{i=1}^M$ which does phase retrieval but complements fail phase retrieval and hence fail norm retrieval~\cite{CCPW}.

\begin{example}
Let $\{\phi _n\}_{n=1}^3 $ and $\{\psi_n\}_{n=1}^3$ be orthonormal bases for $\RR^M$ such that $\{\phi_n\}_{n=1}^3  \cup \{\psi_n\}_{n=1}^3$ is full spark. Consider the subspaces
\begin{alignat*}{2}
\\ &W_1= span(\{\phi_1,\phi_3\}) &\qquad & {W_1}^\perp= span(\{\phi_2\})
\\&W_2= span(\{\phi_2,\phi_3\}) &\qquad &{W_2}^\perp= span(\{\phi_1\})
 \\ &W_3= span(\{\phi_3\}) &\qquad &{W_3}^\perp=span (\{\phi_1,\phi_2\})
 \\ &W_4= span(\{\psi_1\}) &\qquad &{W_4}^\perp =span(\{\psi_2,\psi_3\})
 \\&W_5= span(\{\psi_2\}) &\qquad &{W_5}^\perp=span(\{\psi_1,\psi_3\})
\end{alignat*}
Then $\{W_n\}_{n=1}^5$ allow phase retrieval for $\RR^3$ while the orthogonal complements $\{W_n^ \perp\}_{n=1}^5$ do not.
\end{example}

\begin{corollary}
If $\{\phi_i\}_{i=1}^M$ does phase retrieval and contains an orthonormal basis, then $\{\phi_i^\perp\}_{i=1}^M$	does phase retrieval.
\end{corollary}

\begin{proof}
If $\{\phi_i\}_{i\in I}$ is an orthonormal basis, then $\{\phi_i^\perp\}_{i\in I}$ does norm retrieval. Hence so does the larger set $\{\phi_i^\perp\}_{i=1}^M$. Since $\{\phi_i\}_{i=1}^M$ does phase retrieval, and $\{\phi_i^\perp\}_{i=1}^M$	does norm retrieval, we can conclude the latter does phase retrieval as well which follows from Lemma~\ref{L:perpnormretrieval}.
\end{proof}
The next result gives us a sufficient condition for the subspaces to do norm retrieval. It is enough to check if the identity is in the linear span of the projections in order for the subspaces to do norm retrieval. A similar result in the case of phase retrieval is proved in~\cite{CCJW}.

\begin{proposition}[~\cite{BCCJW}] Let $\{W_i\}_{i=1}^M$ be subspaces of $\RR^N$ with corresponding projections $\{P_i\}_{i=1}^M$. If there exist $a_i\in \RR$ such that $\sum_{i=1}^M a_i P_i=I$, then $\{ P_i\}_{i=1}^M$ does norm retrieval.
\end{proposition}

\begin{proof}
Given $x\in R^N$, then
\begin{align*}
\|x\|^2=\langle x,x\rangle 
&=\big\langle \sum_{i=1}^M a_i P_i x, x\big\rangle=\sum_{i=1}^M a_i\langle P_i x,x\rangle\\ &=\sum_{i=1}^M a_i \langle P_i x,P_i x\rangle=\sum_{i=1}^M a_i\|P_ix\|^2.
\end{align*}
Since for each $i$ the coefficients $a_i$ and $\|P_i x\|$ are known, the collection $\{P_i\}_{i=1}^M$ does norm retrieval.
\end{proof}

A counter example for the converse of the above proposition is given in~\cite{BCCJW} where the authors construct a collection of projections, $P_i$, which do phase retrieval but $I\not\in \text{span } P_i$. Here, we provide another example for the same. We give a set of five vectors in $\RR^3$ which does phase retrieval; however the identity operator is not in the span of these vectors. We need the following theorem that provides a necessary and sufficient condition for a frame to be not scalable in $\RR^3$. Recall that a frame $\{\phi_i\}_{i=1}^M\in \RR^N$ is said to be scalable if there exists scalars $c_i\geq 0, i=1, 2, \ldots, M$ such that $\{c_i\phi_i\}_{i=1}^M$ is a Parseval frame~\cite{GiKaFrEl013}. Later in the next section, we prove that scalable frames always do norm retrieval.

\begin{theorem}\label{T:Gitta}\cite{GiKaFrEl013}
A frame $\phi$ in $\RR^3 - \{0\}$ for $\RR^3$ is not scalabale iff all frame vectors of $\phi$  are contained in an interior of an elliptical conical surface with vertex $0$ and intersecting the corners of a rotated unit cube.
\end{theorem}

\begin{example}
A frame $\{\phi_i\}_{i=1}^5$ in $\RR^3$ which does phase retrieval but 
\[ \sum_{i=1}^5 a_i \phi_i \neq I, \mbox{ for any }a_i \in \RR.\]

Choose five full spark vectors in the cone referred in the previous theorem~\ref{T:Gitta}. These vectors do phase retrieval and hence norm retrieval in $\RR^3$. Now, given $a_i \in \RR$ , $\sum_{i=1}^5 a_i\phi_i=\sum_{i=1}^5 | a_i| (\epsilon_i \phi_i)$ for $ \epsilon_i = \pm 1$. But, $\epsilon_i \phi_i $ is still inside the cone for each $i$. Therefore $\sum_{i=1}^5 |a_i|(\epsilon_i \phi_i) \neq I.$
\end{example}

The next proposition gives a sufficient condition for the complements to do norm retrieval when the subspaces do.


\begin{proposition} If $\{W_i\}_{i=1}^L$ are subspaces of $\RR^N$ with corresponding projections $\{P_i\}_{i=1}^L$ such that $\sum_{i=1}^La_iP_i=I$ and $\sum_{i=1}^La_i\neq 1$. Then $\{ I-P_i\}_{i=1}^L$ does norm retrieval.
\end{proposition}

\begin{proof}
Observe the following
\[ \sum_{i=1}^La_i (I-P_i)=\left(\sum_{i=1}^L a_i\right)I-\sum_{i=1}^La_i P=\left( \sum_{i=1}^L a_i\right) I-I=\left( \sum_{i=1}^La_i - 1\right)I.\]
Let $\alpha=\sum_{i=1}^La_i - 1$ then a short calculation shows $\sum_{i=1}^L\frac{a_i}{\alpha} (I-P_i)=I$. By the previous proposition this shows $\{I-P_i\}_{i=1}^L$ does norm retrieval.
\end{proof}

It is possible that $\sum a_i P_i=I=\sum b_i P_i$ with $\sum a_i=1$ but $\sum b_i\neq 1$, as we will see in the following example.

\begin{example}
Let $\{e_i\}_{i=1}^3$ be an orthonormal basis for $\RR^3$. Now let
\begin{alignat*}{2}
\\ &W_1=span\{e_1\} &\qquad & {W_1}^\perp= span \{ e_2,e_3\}
\\ &W_2=span\{e_2\} &\qquad & {W_2}^\perp= span \{ e_1,e_3\}
\\ &W_3=span\{e_3\} &\qquad & {W_3}^\perp= span \{ e_1,e_2\}
\\ &W_4=span\{e_1,e_2\} &\qquad & {W_4}^\perp= span \{ e_3\}
\\ &W_5=span\{e_1,e_3\} &\qquad & {W_5}^\perp= span \{ e_2\}
\end{alignat*}

Both $\{W_i\}$ and $\{W_i^\perp\}$ do norm retrieval. Let $P_i$ denote the projections on to $W_i$, then $\sum_{i=1}^5a_iP_i=P_1+P_2+P_3+0\cdot P_4+0\cdot P_5=I$ and $\sum_{i=1}^5 b_iP_i=-P_1+0\cdot P_2+0\cdot P_3+P_4+P_5=I$. However, $\sum_{i=1}^5 a_i=3\neq 1 =\sum_{i=1}^5 b_i$.
\end{example}

\section{Classification of Norm Retrieval}\label{s:classifications_normret}

In this section, we give classifications of norm retrieval by projections.
The following theorem in~\cite{HF} uses the span of the frame elements to classify norm retrievable frames in $\RR^N$.

\begin{theorem}(\cite{HF})\label{T:Fard}
A frame $\{\phi_k\}_{k=1}^M\subset \RR^N$ does norm retrieval if and only if for any partition $\{I_j\}_{j=1}^2$ of $[M]:=\{1,2,...,M\}$, $span\ \{\ \phi_k\}_{k\in I_1}^\perp \perp span\ \{\phi_k\}_{k\in I_2}^\perp$.
\end{theorem}

Next, we prove one of the main results of this paper. This is an extension of the previous Theorem~\ref{T:Fard} and it fully classifies the subspaces of $\RR^N$ which do norm retrieval.

\begin{theorem}\label{T:classificationofNormRet}
Let $\{P_i\}_{i=1}^M$ be projections onto subspaces $\{W_i\}_{i=1}^M$ of $\RR^N$.  Then the following are equivalent:
\begin{enumerate}
\item $\{P_i\}_{i=1}^M$ does norm retrieval,
\item Given any orthonormal bases $\{\phi_{i,j}\}_{j=1}^{I_i}$ of $W_i$ and any subcollection $S\subseteq \{(i,j):1\leq i\leq M, 1\leq j\leq I_i\}$ then
\[ span\ \{\phi_{ij}\}_{(i,j)\in S}^\perp
\perp span\ \{\phi_{ij}\}_{(i,j)\in S^c}^\perp,\]
\item For any orthonormal basis $\{\phi_{i,j}\}_{j=1}^{I_i}$ of $W_i$, then the collection of vectors $\{\phi_{i,j}\}_{(i,j)}$ do norm retrieval.
\end{enumerate}
\end{theorem}

\begin{proof}
$(1)\Rightarrow (2)$:  Suppose $x \in span\ \{\phi_{ij}\}_{(i,j)\in S}^\perp $, and $y \in span \ \{\phi_{ij}\}_{(i,j)\in S^c}^\perp$ and let $I=[M]$ then,

\begin{align*}
\|P_i(x+y)\|^2&= \sum_{j=1}^{I_i}|\langle x+y,\phi_{ij}
\rangle|^2\\
 &= \sum_{j\in I\cap I_i}|\langle y,\phi_i\rangle|^2
 + \sum_{j\in I^c\cap I_i}|\langle x,\phi_i\rangle|^2\\
 &= \sum_{j=1}^{I_i}|\langle x-y,\phi_{ij}
\rangle|^2\\
&=\|P_i(x-y)\|^2
\end{align*}
Since $\{P_i\}_{i=1}^M$ does norm retrieval, we have
\[ \|x+y\|^2 = \|x\|^2+\|y\|^2 +2\langle x,y\rangle
=\|x-y\|^2 = \|x\|^2+\|y\|^2-2\langle x,y\rangle,\]
and so $\langle x,y\rangle =0$.
\vskip12pt
$(2)\Rightarrow (1)$: Assume that $\norm{P_ix} = \norm{P_iy}$ for all $1\le i\le M.$ Then, we can find a basis $(\phi_{ij})_{j=1}^{K_i}$ for $W_i$ such that
\[
\abs{\inner{\phi_{ij}}{x}} = \abs{\inner{\phi_{ij}}{y}}.
\]
Denote $A=\set{(i,j)\ :\ \inner{\phi_{ij}}{x}=\inner{\phi_{ij}}{y}}$ and $B=\set{(i,j)\ :\ \inner{\phi_{ij}}{x}=-\inner{\phi_{ij}}{y}}.$ Now we can see that $$(x-y)\bot \mathrm{span}\set{\phi_{ij}\ :\ (i,j)\in A}$$ and also $$(x+y)\bot \mathrm{span}\set{\phi_{ij}\ :\ (i,j)\in B}.$$
By (2), we must have that $\inner{x+y}{x-y}=0,$ which implies that $x$ and $y$ have the same norm.\\
The third equivalence is immediate from the result in Theorem~(\ref{T:Fard}).
\end{proof}

\begin{corollary}
If $\Phi=\{\phi_i\}_{i=1}^M$ does norm retrieval then $\Phi '=\{c_i\phi_i\}_{i=1}^M,$ $c_i\neq 0$ does norm retrieval. Hence all scalable frames do norm retrieval.
\end{corollary}

\begin{proof}
This is an immediate result of Theorem~\ref{T:classificationofNormRet}. Observe the conditions in Theorem~\ref{T:classificationofNormRet} do not depend on the norm of each vector $\phi_i$.
\end{proof}

For the complex case we have:

\begin{proposition}
If $\{P_i\}_{i=1}^M$ does norm retrieval, then whenever we choose orthonormal bases $\{\phi_{i,j}\}_{j=1}^{I_i}$ of $W_i$ and any subcollection $S\subseteq \{(i,j):1\leq i\leq M, 1\leq j\leq I_i\}$ then
\[ x \perp span\ \{\phi_{ij}\}_{(i,j)\in S} \mbox{ and } y \perp span\ \{\phi_{ij}\}_{(i,j)\in S^c}\mbox{ implies } Re\langle x,y\rangle =0.\]
\end{proposition}

\begin{proof}
Given $x,y$ as above,
\[ |\langle x+y,\phi_{ij}\rangle|=|\langle x-y,\phi{ij}\rangle|, \mbox{ for all (i,j)}.\]
Since our vectors do norm retrieval, we have
\[ \|x+y\|^2 = \|x\|^2+\|y\|^2 +2\langle x,y\rangle =\|x-y\|^2 = \|x\|^2+\|y\|^2-2\langle x,y\rangle,\]
and so $Re\langle x,y\rangle =0$.
\end{proof}

We use Theorem~\ref{T:classificationofNormRet} to give a simple proof of a result in~\cite{CCJW} which has a very complicated proof in that paper.

\begin{corollary}\label{C:Normret_orthogonal}
If $\{\phi_i\}_{i=1}^N$ do norm retrieval in $R^N$, then the vectors are orthogonal.
\end{corollary}

\begin{proof}
Assume $\|\phi_i\|=1$ and that $\phi_j$ is not orthogonal
so span $\{\phi_i\}_{i\not= j}$.  Choose $x \perp
a_i$ for all $i\not= j$.  Let $y= x-\langle x,a_j\rangle a_j$.  Now,
\[ \langle a_j,y\rangle = \langle a_j,x\rangle - \langle x,
a_j\rangle\langle a_j, a_j\rangle=0.\]
Let $I=\{i:i\not= j\}$.  Then
\[ x \perp span\ \{a_i\}_{i\in I}\mbox{ and }
y\perp a_j,\]
but
\[ \langle x,y\rangle = \langle x,x\rangle -\langle x,a_j\rangle \langle x,a_j\rangle = 1-|\langle x,a_j\rangle|^2
\not= 0,\]
contradicting the theorem.
\end{proof}

\begin{corollary}
Consider a frame $\Phi= \{\phi_i\}_{i=1}^M$. The followings are equivalent:
\begin{enumerate}
\item $\Phi$ does norm retrieval.
\item  For $ i=1,2,\ldots,M$ if $W_1=span \{\phi_i\}_{i \in I}$ and $W_2=span \{\phi_i\}_{i \in I^c}$ then, ${W_1}^\perp \subseteq {W_2}$.
\end{enumerate}
\end{corollary}

\begin{proof}
By Theorem~\ref{T:classificationofNormRet}, it follows that $\Phi$ does norm retrieval if and only if ${W_1}^\perp \perp {W_2}^\perp$. This happens if and only if ${W_1}^\perp \subseteq {W_2}$. Hence the proof.
\end{proof}

Both phase retrieval and norm retrieval are preserved when applying projections to the vectors. Also, phase retrieval is preserved under the application of any invertible operator (refer to~\cite{BCCJW} for details). This is not the case with norm retrieval, in general. We prove this in the next corollary.

\begin{corollary}
Norm retrieval is not preserved under the application of an invertible operator, in general.
\end{corollary}

\begin{proof}
Let $\phi=\{\phi_i\}_{i=1}^N$ be linearly independent vectors in $R^N$ which are not orthogonal. Then by Corollary~\ref{C:Normret_orthogonal}, $\Phi$ cannot do norm retrieval. But there exists an invertible operator $T$ on $R^N$ so that $\{T\phi_i\}_{i=1}^N$ is an orthonormal basis and so does norm retrieval.
\end{proof}

However, we note that unitary operators, which are invertible, do preserve norm retrieval.  

%

The following corollary about Parseval frames also holds in the infinite dimensional case with the same proof.
\begin{corollary}
If $\Phi$ is a Parseval frame, it does norm retrieval.  Hence, if we partition $\Phi$ into two disjoint sets, and choose a vector orthogonal to each set, then these vectors are orthogonal.
\end{corollary}

\begin{proof}
Let $\Phi =\{\phi_i\}_{i\in I}$ be a Parseval frame and let $J\subseteq I$. Let $T$ be its analysis operator. If $x\perp \{\phi_i\}_{i\in J}$ and $y\perp \{\phi_i\}_{i\in J^c}$. Then $Tx=(\ipt{x, \phi_i})$ and $Ty=(\ipt{y, \phi_i})$ do not have any nonzero coordinates in common. So $Tx\perp Ty$. Since, the analysis operator of a Parseval frame is an isometry, we have $x\perp y$.
\end{proof}

A classic result in frame theory is that a Parseval frame $\{\phi_i\}_{i\in I}$ has the property that if $\phi_j \notin W= span_{i\neq j}\{\phi_i\}$ then $\phi_j \perp W.$  It turns out that a much
more general result holds.

\begin{corollary}
Let $\{\phi_i\}_{i=1}^N$ be a Parseval frame in $\mathbb{R}^M$. For $I\subseteq[N]$, let $W_I=span \{\phi_i\}_{i\in I}$ and $W_{I^c}=span\{\phi_i\}_{i\in {I^c}}$. If $W_I\cap W_{I^c} =\{0\}$, then $W_I \perp W_{I^c}$.
\end{corollary}

\begin{corollary}
If $\Phi= \{\phi_i\}_{i=1}^N$ is a frame for $R^M$ with frame operator $S$ which does norm retrieval, then for every
$I\subset \{1,2,\ldots,N\}$, if $x\perp span\ \{\phi_i\}_{i\in I}$ then $x \in span\ \{S^{-1}\phi_i\}_{i\in I^c}$.  In particular, if
$\Phi$ is a Parseval frame then $x \in span\ \{\phi_i\}_{i\in I^c}$.
\end{corollary}

\begin{proof}
Given x as in the corollary,
\begin{align*}
x &= \sum_{i=1}^N \langle x,\phi_i\rangle S^{-1}\phi_i\\
&= \sum_{i\in I^c}\langle x,\phi_i\rangle S^{-1}\phi_i.
\end{align*}
\end{proof}

We next provide a classification of norm retrieval using Naimark's theorem. It turns out that every frame can be scaled to look similar to Naimark's theorem.


\begin{proposition}
If $\{\phi_i\}_{i=1}^M$ is a frame with Bessel bound $B$ on $\RR^N$, then $\RR^N \subset \ell_2^{2M-1}$ with orthonormal basis $\{e_i\}_{i=1}^{2M-1}$ so that the orthogonal projection onto $\RR^N$ satisfies:  $Pe_i=\phi_i$ for every $i\in [M]$.
\end{proposition}

\begin{proof}
Let $\{g_i\}_{i=1}^N$ be the eigenbasis for the frame with respective eigenvalues $1=\lambda_i \ge \lambda_2 \ge \cdots \ge \lambda_N$.  For $M+1\le M+i \le 2M-1$ let
\[ \phi_{M+i} = \sqrt{1-\lambda_{i+1}}e_{i+1}.\]
Then $\{\phi_i\}_{i=1}^{2M-1}$ is a Parseval frame.  So $\RR^N \subset \ell_2(2M-1)$ with orthonormal basis $\{e_i\}_{i=1}^{2M-1}$ and the projection down to $\RR^N$ satisfies $Pe_i = \phi_i$ for all $i\in [2M-1]$.
\end{proof}

\begin{theorem}
Let $\Phi=\{\phi_i\}_{i=1}^M$ be a frame for $\RR^N$.  The following are equivalent:
\begin{enumerate}
\item $\Phi$ does norm retrieval.
\item $\RR^N \subset \ell_2^{2M-1}$ with orthonormal basis $\{e_i\}_{i=1}^{2M-1}$ and for every $x\in \RR^N$, if $|\langle x,e_i\rangle|$ is known for $i=1,2,\ldots,M$ then
\[ \|\sum_{i=M+1}^{2M-1}\langle x,e_i\rangle e_i\|^2 = \sum_{i=M+1}^{2M-1}|\langle x,e_i\rangle|^2, \mbox{ is known }.\]
In other words, if $x,y \in \RR^N$ and
\[|\langle x,e_i\rangle|=|\langle y,e_i\rangle| \mbox{ for all }i=1,2,\ldots,M,\]
then
\[ \|\sum_{i=M+1}^{2M-1}\langle x,e_i\rangle e_i\| = \|\sum_{i=M+1}^{2M-1}\langle y,e_i\rangle e_i\|.\]
\end{enumerate}
\end{theorem}

\begin{proof}
We may assume $\RR^N \subset \ell_2^{2M-1}$. Let $\{e_i\}_{i=1}^{2M-1}$ be an orthonormal basis for $\ell_2^{2M-1}$ and the projection onto $\RR^N$ satisfies $Pe_i=\phi_i$ for $i=1,2,\ldots,M$. Now, $\Phi$ does norm retrieval if and only if for any $x \in \RR^N$, knowing $|\langle x,\phi_i\rangle|$ gives us $\|x\|$. But
\[\langle x,\phi_i\rangle = \langle x,Pe_i\rangle = \langle Px,e_i\rangle = \langle x,e_i\rangle.\]
Now, knowing $|\langle x,e_i\rangle|$ for $i=1,2,\ldots,M$ means knowing $\|x\|$.
But:
\[ \|x\|^2=\sum_{i=1}^M|\langle x,e_i\rangle|^2+\|\sum_{i=M+1}^{2M-1}\langle x,e_i\rangle e_i\|^2\].
\end{proof}

%
%
%
%
%
%
%

\end{document}